\documentclass[12pt]{amsart} 

\usepackage[utf8]{inputenc} 
\usepackage{amssymb,caption}
\usepackage{enumerate}

\usepackage{tikz}
\usetikzlibrary{calc,intersections,shapes.geometric,through}
\def\comment{}
\def\endcomment{}

\usepackage{geometry} 
\newtheorem{lemma}{Lemma}
\newtheorem{prop}[lemma]{Proposition}
\newtheorem{theorem}[lemma]{Theorem}
\theoremstyle{definition}

\newtheorem{defn}[lemma]{Definition}
\newtheorem{remark}[lemma]{Remark}
\newtheorem*{notation}{Notation}
\newcommand{\A}{\mathbb{A}}
\renewcommand{\P}{\mathbb{P}}

\newcommand{\circleThroughThreePoints}[4]{
\coordinate (middle1) at ($(#1)!.5!(#2)$);
\coordinate (middle2) at ($(#2)!.5!(#3)$);
\coordinate (aux1) at ($(middle1)!1!90:(#2)$);
\coordinate (aux2) at ($(middle2)!1!90:(#3)$);
\coordinate (center) at ($(intersection of middle1--aux1 and middle2--aux2)$);
\node at (center) [draw, circle through=(#1), name path=#4] {};
}

\newcommand{\circleThroughThreePointsWLabel}[6]{
\coordinate (middle1) at ($(#1)!.5!(#2)$);
\coordinate (middle2) at ($(#2)!.5!(#3)$);
\coordinate (aux1) at ($(middle1)!1!90:(#2)$);
\coordinate (aux2) at ($(middle2)!1!90:(#3)$);
\coordinate (center) at ($(intersection of middle1--aux1 and middle2--aux2)$);
\node at (center) 
[draw=green,label=#5:$#6$, circle through=(#1), name path=#4] {};
}

\newcommand{\circleThroughPointTangentWLabel}[6]{
\coordinate (middle1) at ($(#1)!.5!(#2)$);
\coordinate (aux1) at ($(middle1)!1!90:(#2)$);
\coordinate (aux2) at ($(#2)!1!90:(#3)$);
\coordinate (center) at ($(intersection of middle1--aux1 and #2--aux2)$);
\node at (center) [draw=green,label=#5:$#6$, circle through=(#1), name path=#4] {};
}

\def\lijn(#1,#2){\langle #1, #2 \rangle}

\title{Circle incidence theorems}
\author{J. Chris Fisher}
\address{2616 Edgehill Road\\                                     
Cleveland Heights, OH 44106-2806  \\
USA}
\email{fisher@math.uregina.ca}
\author{Eberhard M. Schr\"oder} 
\address{FB Mathematik der Universit\"at\\
Bundesstrasse 55\\
D 20146 Hamburg\\
Germany}
\email{eberhard.schroeder@gmx.net}
\author{Jan Stevens}
\address{Mathematical Sciences\\
University of Gotheburg\\
Chalmers University of Technology\\
SE 412 96 Gothenburg\\
Sweden}
\email{stevens@chalmers.se}

\begin{document}
\maketitle

Given a triangle, there are 
unexpected triples of lines that pass
through one point; e.g, the three medians, altitudes, and angle 
bisectors are all concurrent.
Larry Hoehn discovered a remarkable 
concurrence theorem about pentagons,
illustrated in Figure \ref{fivecirclefig}, see \cite{FHS1}.
In this note we prove a generalization to $n$-gons.

Let $A_1$, \dots, $A_n$ be $n$ points in the plane, no three on a
line, and such that the lines $l_{i+1}=\langle A_i,A_{i+2}\rangle$
and  $l_{i}=\langle A_{i-1},A_{i+1}\rangle$ are not parallel,
where we consider the indices modulo $n$. 
Let $B_{i,i+1}$ be the
intersection point of $l_i$ and $l_{i+1}$. Through the three
points $A_i$, $B_{i,i+1}$ and $A_{i+1}$ passes a unique
circle $c_{i,i+1}$. Let $g_i$ be the radical axis of 
the two consecutive circles $c_{i-1,i}$ and $c_{i,i+1}$.

\begin{theorem}[\cite{FHS1}]
Given five 
points $A_1$, \dots, $A_5$ in the plane the five radical axes
$g_1$, \dots, $g_5$, constructed as above
are concurrent or parallel.
\end{theorem}

\begin{figure}
\comment
\begin{tikzpicture}
[thick,scale=1.3,punt/.style={circle,fill=red,inner sep=1pt},
punk/.style={circle,fill=green,inner sep=1pt}]

\coordinate [label=above:$A_1$,punt] (A)  at (-0.6,2.8) ;
\coordinate [label=above left:$A_2$,punt] (B) at (0,0) ;
\coordinate [label=above right:$A_3$,punt] (C) at (3,0.5) ;
\coordinate [label=above left:$A_4$,punt] (D) at (3.5,3) ;
\coordinate [label=right:$A_5$,punt] (E) at (1.5,4.5)  ;

\draw [red] (A) -- (C) -- (E) -- (B) --(D) -- (A);

\coordinate [punk](A23) at (intersection of A--C and B--D);
\coordinate [punk](A34) at (intersection of B--D and C--E);
\coordinate [punk](A45) at (intersection of C--E and D--A);
\coordinate [punk](A15) at (intersection of D--A and E--B);
\coordinate [punk](A12) at (intersection of E--B and A--C);
\circleThroughThreePoints{B}{A23}{C}{c23};
\circleThroughThreePoints{C}{A34}{D}{c34};
\circleThroughThreePoints{D}{A45}{E}{c45};
\circleThroughThreePoints{B}{A12}{A}{c12};
\circleThroughThreePoints{A}{A15}{E}{c15};

\coordinate (AE) at ($(A)+2*($(B)-(D)$)$) ;
\coordinate (BC) at ($(B)+2*($(A)-(D)$)$) ;
\coordinate (CD) at ($(C)+2*($(B)-(E)$)$) ;
\coordinate (ED) at ($(E)+2*($(A)-(C)$)$) ;

\coordinate (a) at (intersection of B--BC  and E--ED);
\coordinate (b) at (intersection of C--CD  and A--AE);
\coordinate [punt ](M) at (intersection of  A--a and B--b);
\draw[blue] (E)--(M)--(A) (B)--(M)--(C) (M)--(D);

\end{tikzpicture}
\endcomment
\caption{The 5-circle theorem}\label{fivecirclefig}
\end{figure}
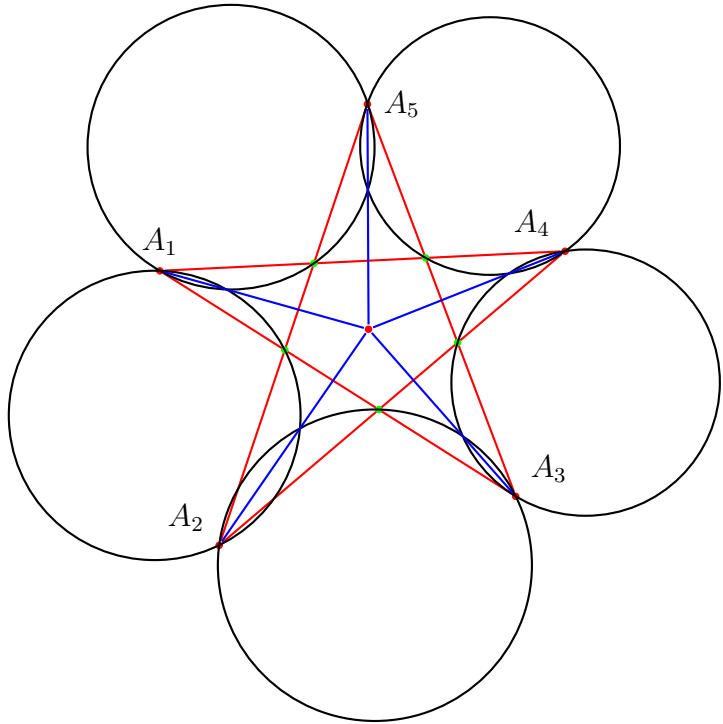

We use the terminology that lines 
\textit{lie in a pencil}
if they  are concurrent or parallel.
For $n\geq 6$ the radical axes in general  do not lie in a pencil.
For $n=6$ we show that 
it is necessary and sufficient that  the six points $B_{i,i+1}$ lie on a conic.
This is  equivalent
to the 
condition that the three lines $\lijn(A_i,A_{i+3})$ lie in a pencil. 
In fact, 
the initial six points have to be in a special position for just 
three consecutive axes to lie in a pencil:
Fisher, Hoehn and Schr\"oder showed that 
this condition
implies that than the remaining three axes lie in the same pencil \cite{FHS2}.
Our main result generalizes this to $n>6$.
\begin{theorem}
Let $A_1$, \dots, $A_n$ be $n$ points in the plane, no three on a
line, and such that the lines $l_{i-1}=\langle A_{i-1},A_{i+1}\rangle$ and 
$l_{i+1}=\langle A_i,A_{i+2}\rangle$ intersect in a point $B_{i,i+1}$ 
\textup(indices considered modulo $n$\textup). 
Let $c_{i,i+1}$ be the circle through $A_i$, $B_{i,i+1}$ and $A_{i+1}$,
and let $g_i$ be the radical axis of the circles $c_{i-1,i}$ and $c_{i,i+1}$.
\\
If the lines $g_1$, $g_2$, \dots, $g_{n-3}$ lie in a pencil, 
then the remaining three radical axes $g_{n-2}$, $g_{n-1}$ and $g_n$
lie in the same pencil.
\end{theorem}

We prove the theorem under weaker assumptions 
and in a more general setting. 
As shown in \cite{FHS2}, the theorem is a result in affine geometry:
a radical axis $g_i$ can be constructed by drawing parallel lines.

We can relax the condition that no three points lie on a line.  In fact,
the theorem continues  to hold in certain limiting cases, if the elements
of the construction are suitably reinterpreted. We make one case for $n=5$
explicit 
for later use.

\section{Preliminaries}
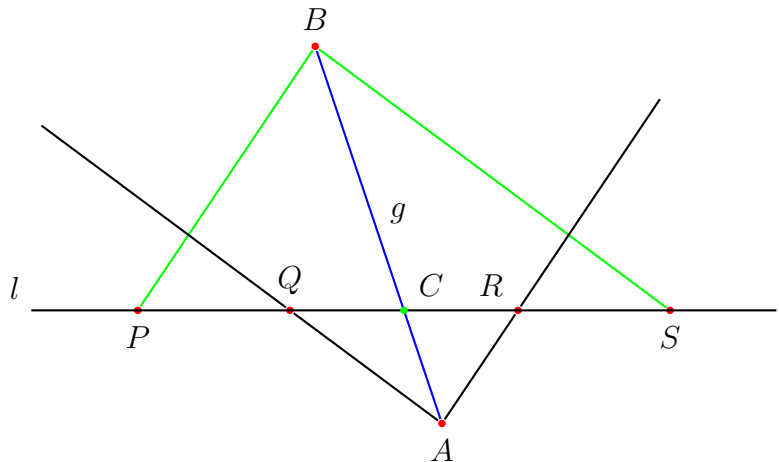
\begin{figure}[b]
\comment
\begin{tikzpicture}
[thick,punt/.style={circle,fill=red,inner sep=1pt},
punk/.style={circle,fill=green,inner sep=1pt}]
\coordinate [label=below:$P$,punt] (P) at (0,0) ;
\coordinate  [label=above:$Q$,punt] (Q) at (2,0) ;
\coordinate [label=above left:$R$,punt] (R) at (5,0) ;
\coordinate  [label=below:$S$,punt] (S) at (7,0) ;
\coordinate [label=below:$A$,punt] (A) at (4,-1.5) ;
\coordinate [label=above:$B$,punt] (B) at (intersection of P--{$(P)+(R)-(A)$}
and S--{$(S)+(Q)-(A)$});

\draw[green] (P) -- (B) -- (S); 
\draw ($(Q)+0.7*($(B)-(S)$)$)--(Q) -- (A) -- (R)--($(R)+0.8*($(B)-(P)$)$) 
($(P)!1.2!(S)$) -- ($(S)!1.2!(P)$) node[black,above left]{$l$};
\draw[blue]  (B) --  node[black,above right] {$g$} (A);

\node [label=above right:$C$,punk] 
    at (intersection of A--B and Q--R)  {};
\end{tikzpicture}
\endcomment
\caption{Construction of the axis}\label{axisfig}
\end{figure}

We work in the affine plane $\A^2(k)$ over an arbitrary field $k$, which we 
view as embedded in $\P^2(k)$. All lines considered are projective lines.
Two lines (different from the line at infinity) are
parallel if their intersection point is a point at infinity.
A general reference for this section is the book \cite{Ev}.
\begin{defn}\label{axisdef}
Let $(P,Q)$ and $(R,S)$ be two 
pairs of finite
points on a line $l$; it is allowed that $P=Q$ or $R=S$, 
but neither $R$ nor $S$
may coincide with $P$ or $Q$. Let $A\notin l$ be a finite point.
Denote by $l_P$ be the line through $P$ that is parallel to the line 
$\langle A,R \rangle$ and take $l_S\parallel \langle A,Q \rangle$ through
$S$. Set $B=l_P\cap l_S$. The line $g=\langle A,B\rangle$ is the \textit{axis}
of the configuration, see figures \ref{axisfig} and \ref{infaxisfig}.  
\end{defn}

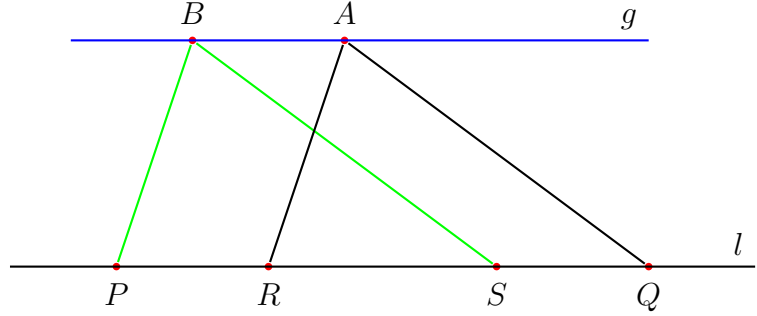
\begin{figure}
\comment
\begin{tikzpicture}
[thick,punt/.style={circle,fill=red,inner sep=1pt},
punk/.style={circle,fill=green,inner sep=1pt}]
\coordinate [label=above:$B$,punt] (B) at (1,3) ;
\coordinate [label=below:$P$,punt] (P) at (0,0) ;
\coordinate [label=below:$R$,punt] (R) at (2,0) ;
\coordinate [label=below:$S$,punt] (S) at (5,0) ;
\coordinate [label=below:$Q$,punt] (Q) at (7,0) ;
\coordinate  [label=above:$A$,punt] (A) at (3,3) ;
\draw [green] (P) -- (B) -- (S); 
\draw (Q) -- (A) -- (R)
($(Q)!1.2!(P)$) -- ($(P)!1.2!(Q)$) node[black,above left]{$l$}; 
\draw[blue]   ($(A)!1.8!(B)$) -- ($(B)!3!(A)$) node[black,above left] {$g$} ;
\end{tikzpicture}
\endcomment
\caption{Axis parallel to the line}\label{infaxisfig}
\end{figure}

The difference $P-Q$ of two  points in the affine plane 
is a well defined vector in the
associated vector space.
For points $P, Q, R, S$ on a line with $R\neq S$, 
the vector $P-Q$ is a scalar multiple of the vector $R-S$, so
the ratio
$\frac{P-Q}{R-S}$ 
is an element of the ground field $k$. We use the convention that 
$\frac{P-Q}{P-R}=1$ if $P$ lies at infinity and $Q$ and $R$ are distinct
finite points.

\begin{lemma}\label{lem_C}
\label{intsecptlemma}
The intersection point $C=g\cap l $ is determined by the 
equivalent conditions
\[
\frac {C-Q}{C-R}=\frac{Q-S}{R-P}\;,
\]
which in case $P\neq Q$ is equivalent to
\[
\frac{C-Q}{C-P}=\frac{R-Q}{R-P}\frac{S-Q}{S-P}\;
\]
and to 
\[
\frac{C-S}{C-R}=\frac{Q-S}{Q-R}\frac{P-S}{P-R}\;
\]
in case $R\neq S$.
\end{lemma}
\begin{notation}
We denote the point so determined by $C=[P,Q\mid R,S]$.
\end{notation}

The lemma can be proved by direct computation.
It also follows (if the four points $P$, $Q$, $R$ and $S$ are all 
distinct) from \cite[Lemma 1]{FHS2} and its corollary, which moreover 
establish that the above affine definition of the axis gives the radical axis
of circles as in figure \ref{axispluscirclefig}, 
in the context of general affine metric planes.

\begin{remark}\label{euclrem}
For the 
euclidean plane these properties can easily be
established with geometric arguments.
To prove the lemma we use similarity of triangles 
in figure \ref{axisfig}, in case $C$ is a finite point. 
We have that $\triangle BCP \sim \triangle ACR $ and 
$\triangle BCS \sim \triangle ACQ $. Therefore
\[
\frac{C-P}{C-R}=\frac{C-B}{C-A}=\frac{C-S}{C-Q}\;.
\]
It follows that 
\[\frac{R-P}{C-R}=\frac{C-P}{C-R}-1=\frac{C-S}{C-Q}-1=\frac{Q-S}{C-Q}\;.
\]
In the case that
$C$ lies at infinity (figure \ref{infaxisfig}) we have
$R-P=B-A=Q-S$.

To find the axis as radical axis we add circles to the figure
(see figure \ref{axispluscirclefig}).  
Let $c_1$ be the circle
through $A$, $P$, $Q$ and $c_2$ the circle through $A$, $R$, $S$.
If $P=Q$, then $c_1$ is the circle through $A$ which is tangent to
the line $l$ in the point $P=Q$; if $R=S$, the circle $c_2$ is 
tangent to $l$. 
Consider also the circle $c_3$ through $A$, $Q$ and $R$.
Then $c_1$ and $c_3$ intersect in $A$ and $Q$, so the line 
$\langle A,Q \rangle$ is the radical axis of $c_1$ and $c_3$.
The parallel line $l_S$ is the locus of points for which
the power w.r.t.~$c_1$ has constant difference with the power w.r.t.~%
$c_3$, the difference being $(S-P)\cdot(S-Q)
-(S-Q)\cdot(S-R)=(S-Q)\cdot(R-P) $. 
The line $l_P$ is 
the locus where the power w.r.t.~$c_2$ differs from the power
w.r.t.~$c_3$ by the same quantity, as $(P-S)\cdot(P-R)
-(P-R)\cdot(P-Q)=(P-R)\cdot(Q-S)$.
Therefore the intersection point $B=l_S\cap l_P$ lies on the radical
axis of $c_1$ and $c_2$, so this radical axis is the axis  
$g=\langle A,B \rangle$.

\begin{figure}
\comment
\begin{tikzpicture}
[thick,punt/.style={circle,fill=red,inner sep=1pt},
punk/.style={circle,fill=green,inner sep=1pt}]
\clip (-2.5,-4) rectangle (8.5,4);
\coordinate [label=above left:$P$,punt] (P) at (0,0) ;
\coordinate  [label=above:$Q\;\;$,punt]  (Q) at (2,0) ;
\coordinate  [label=above left:$R$,punt] (R) at (5,0) ;
\coordinate  [label=above right:$S$,punt] (S) at (7,0) ;
\coordinate [label=below:$A\;\;$,punt]  (A) at (4,-1.5) ;
\coordinate [label=above:$B$,punt] (B) at (intersection of P--{$(P)+(R)-(A)$}
and S--{$(S)+(Q)-(A)$});

\draw [brown] (P) -- (B) -- (S) ;
\draw ($(Q)+0.7*($(B)-(S)$)$)--(Q) -- (A) -- (R)--($(R)+0.8*($(B)-(P)$)$) 
($(P)!1.2!(S)$) -- ($(S)!1.2!(P)$);
\draw[blue] ($(B)!1.5!(A)$)--(B);

\circleThroughThreePointsWLabel{A}{P}{Q}{c1}{130}{c_1};
\circleThroughThreePointsWLabel{A}{R}{S}{c2}{0}{c_2};
\circleThroughThreePointsWLabel{A}{Q}{R}{c3}{130}{c_3};
\node [label=above:$\;\;C$,punk] 
    at (intersection of A--B and Q--R)  {};

\end{tikzpicture}
\endcomment
\caption{}\label{axispluscirclefig}
\end{figure}

In the situation of figure \ref{infaxisfig} the center of the circle
$c_1$ lies on the perpendicular bisector of $PQ$, which is also the
perpendicular bisector of $RS$, on which the center of $c_2$ lies. 
Therefore the
radical axis is parallel to $l$ and $B$ lies on it.
\end{remark}

\begin{lemma}\label{invo}
Given $C$ and $(R,S)$ on $l$, the map $\gamma\colon l\to l$, sending
$X\in l$ to the point $\gamma(X)$, determined by $C=[X,\gamma(X)\mid R,S]$
is an involutive projectivity.
\end{lemma}
\begin{proof}
To find $\gamma(X)$ we choose a point $A\notin l$ and
draw the line $l_X$ through
$X$, parallel to $\langle R, A\rangle$ (see figures \ref{axisfig}
and \ref{infaxisfig},
reading $X$ and $\gamma(X)$ for $P$ and $Q$). 
It intersects the line $g$ in a point
$Y$. Through $Y$ we draw the line $l_S=\langle Y,S\rangle$.
Then we draw  a line $m$ through $A$ parallel to $l_S$ and define 
$\gamma(X) = l\cap m$. This construction can be described
as first projecting the line $l$ from the point at infinity on the line 
$\lijn(A,R)$  onto the line $g$, 
then projecting $G$ from $S$ onto the line $l_\infty$ at infinity
and finally
projecting $l_\infty$ onto $l$ from 
$A$. This shows that the map $\gamma$ is a projectivity.

That $\gamma^2=\text{id}$ can be seen from the formulas in lemma 
\ref{intsecptlemma}
or by observing that 
$\gamma$ interchanges $R$ with $S$, and $C$  
with the point at infinity on the line $l$.
\end{proof}

\begin{figure}
\comment
\begin{tikzpicture}
[thick,punt/.style={circle,fill=red,inner sep=1pt},
punk/.style={circle,fill=green,inner sep=1pt}]
\clip (-3,-3) rectangle (8,4);
\coordinate [label=above:$B$,punt] (B) at (1,3) ;
\coordinate [label=above left:$P$,punt] (P) at (0,0) ;
\coordinate  [label=below right:{$\quad A$},punt] (R) at (3,0) ;
\coordinate [label=above right:$S$,punt] (S) at (6,0) ;
\coordinate (RP) at ($(R)+0.8*($(B)-(P)$)$) ;
\coordinate (RS) at ($(R)+0.7*($(B)-(S)$)$) ;

\draw [brown] (P) -- (B) -- (S) ;
\draw
(RP) node[above]{$l_R$} -- (R) -- (RS) node[above]{$l_Q$}
($(P)!1.2!(S)$) -- ($(S)!1.2!(P)$);

\circleThroughPointTangentWLabel{P}{R}{RS}{c1}{150}{c_1}
\circleThroughPointTangentWLabel{S}{R}{RP}{c2}{-40}{c_2}

\draw[blue]   (B) --(R)-- ($(B)!2!(R)$) node[black,above] {$g$} ;

\end{tikzpicture}
\endcomment
\caption{}\label{degencirc}
\end{figure}

\begin{remark} Given the involution $\gamma\colon l \to l$ the point $C$
is determined as the image of the point  at infinity on the line $l$.
\end{remark}
 
\begin{remark}\label{axis_for-A-in-l}
The point $C$ on $l$ is determined by the unordered pairs $(P,Q)$ 
and $(R,S)$, independent of the point $A$ outside the line. 
We have emphasized the construction
using a particular choice of points ($Q$ and $R$) connected to $A$, as
the construction with the points $A_1$, \dots, $A_n$ naturally leads
to this situation: the line $l=l_i$ is determined by the points
points $P=A_{i-1}$ and $S=A_{i+1}$, while $Q=B_{i-1,i}$ and $R=B_{i,i+1}$
arise as intersection points of $l$ with the lines $l_{i-1}=
\lijn(A_{i-2},A_i)$ and $l_{i+1}=\lijn(A_i,A_{i+2})$.
This extra structure makes it possible to define the axis 
if $A_i\in l_i$; in such a case there would be no
involution on the line $l_i$.

Let $A$ be a point on the line $l=\lijn(P,S)$, different from $P$ and $S$
and let $l_Q$ and $l_R$ be two lines through $A$. Denote by $B$
the intersection point  of the line $l_P$ through $P$, parallel to $l_R$
and $l_S$ through $S$, parallel to $l_Q$. We define the axis of this 
configuration as  the line $\lijn(A,B)$.
In the case of the Euclidean plane it is the radical axis of the circle
though $P$, tangent to $l_Q$ in $A$, and the circle through $S$, tangent 
to $l_R$ in $A$. The proof of   Remark \ref{euclrem} extends to this situation,
with the  circle $c_3$ reduced  to the point
$A=Q=R$ (compare figure \ref{degencirc} with figure \ref{axispluscirclefig}).
\end{remark}

\section{An $n$-axes theorem}
We now formulate our main theorem.
\begin{theorem}\label{mainthm}
Let $A_1$, \dots, $A_n$ be a sequence of $n\geq5$ distinct points in $\A^2(k)$, 
and define
$l_{i}=\langle A_{i-1},A_{i+1}\rangle$ \textup(indices considered modulo $n$%
\textup).   
Assume that
\newcounter{enumval}
\begin{enumerate}[\rm(i)]
\item \label{ass1} $A_i\notin l_{i-2}, l_i, l_{i+2}$,
\item \label{ass2} $l_{i-1}\neq l_{i+1}$,
\item \label{ass3} $l_i \nparallel l_{i+1}$,
\setcounter{enumval}{\value{enumi}}
\end{enumerate}
and set $B_{i,i+1}=l_i \cap l_{i+1}$,
$C_i=[A_{i-1},B_{i-1,i}|B_{i,i+1},A_{i+1}]$, and, finally, let
$g_i=\lijn(A_i,C_i)$ be the axis through $A_i$. 
If the $n-3$ axes $g_1$, $g_2$, \dots, $g_{n-3}$ lie in a pencil, 
then the remaining three  axes $g_{n-2}$, $g_{n-1}$, $g_n$
 lie in the same pencil.
\end{theorem}
As $A_i\in l_{i+1}=\lijn(A_i,A_{i+2})$ but $A_i\notin l_i$ by assumption
\eqref{ass1}, we have that $l_i\neq l_{i+1}$ and therefore 
assumption \eqref{ass3} guarantees the existence
of the point $B_{i,i+1}$ as a well-defined finite point.

By \eqref{ass1} and \eqref{ass2} the points $A_{i-1}$, $B_{i-1,i}$,
$B_{i,i+1}$ and $A_{i+1}$ are four distinct points on the line $l_i$
and $A_i$ is a point outside, so that the axis $g_i$ is
defined.
The condition  $l_{i-1}\neq l_{i+1}$ means that 
$A_{i-2}$, $A_i$ and $A_{i+2}$ are not collinear.
It is therefore equivalent to each of the conditions
$A_{i-2}\notin l_{i+1}$ and $A_{i+2}\notin l_{i-1}$.
Therefore the  assumptions 
\eqref{ass1} and \eqref{ass2}
can be replaced by
\begin{enumerate}[\rm(i)]
\setcounter{enumi}{\theenumval}
\item \label{ass4} $A_i\notin  l_{i-3}, l_{i-2}, l_i, l_{i+2}, l_{i+3}$.
\end{enumerate}
In particular this means that for $n\leq 6$, 
\eqref{ass1} and \eqref{ass2} together are equivalent to the
condition that no three points are collinear. Therefore the Theorem holds
for $n=5 $ and $n=6$ by the results of \cite{FHS2}.
\begin{defn}
We call the common (finite or infinite) point of the pencil 
$\{g_i\}$
the \textit{center} of the sequence $A_1$, \dots, $A_n$.
\end{defn}

\section{A degenerate case of the 5-axes theorem}\label{degenerate}
The $5$-axes theorem states that for five points $A_1$,\dots, $A_5$
in the plane, no three collinear, and $\langle A_{i-1}, A_{i+1}\rangle
\nparallel \langle A_{i}, A_{i+2}\rangle $, the five axes $g_1$, \dots,
$g_5$ lie in a pencil.
Motivated partly because they will be required later, but also because they are 
themselves of some interest,
we study in this section some special and limiting cases. We first consider
when the center is a point at infinity.
More generally, we
investigate the relationship of the center to the 
position of the initial five points.

\begin{prop}\label{centerprop}
Consider four points $A_1$, $A_2$, $A_3$ and $A_4$
in an affine plane $\A^2(k)$, such that no three are collinear and such that
$l_2=\langle A_1, A_3 \rangle$ is not parallel to  
$l_3 = \langle A_2, A_4 \rangle $.
A point $A_5$ in the plane, such that the assumptions of the 5-axes theorem are 
satisfied \textup(i.e., $A_5$ does not lie on a line $\langle A_i,A_j\rangle$, 
while $l_i\nparallel l_{i+1}$ for all $i\neq 2$\textup)
determines a center $M$ in the extended plane $\P^2(k)$. The correspondence
$A_5 \mapsto M$ is the restriction of a projective transformation
$\P^2(k) \to \P^2(k)$. In particular, 
the locus of points $A_5$ for which $M$ is a point at infinity 
\textup(i.e., for which the axes are parallel\textup) is a line.
\end{prop}

\begin{proof}
This is a computation. We construct the axis $g_i$ from the intersection
point $E_i$ of the line through $A_{i-1}$, parallel to $\lijn(A_{i},B_{i,i+1})=
\lijn(A_i,A_{i+2})$ , with the parallel to $\lijn(A_{i},B_{i,i-1})=
\lijn(A_i,A_{i-2})$ through $A_{i+1}$, see figure \ref{affinefiveaxes}.

\begin{figure}
\comment
\begin{tikzpicture}
[thick,scale=3,punt/.style={circle,fill=red,inner sep=1pt},
punk/.style={circle,fill=green,inner sep=1pt},
punb/.style={circle,fill=blue,inner sep=1pt}]
\coordinate [label=below left:$A_1$,punt] (A) at (0,0) ;
\coordinate [label=right:$A_2$,punt] (B) at (2,1.2) ;
\coordinate [label=left:$A_3$,punt] (C) at (0,1) ;
\coordinate [label=below right:$A_4$,punt] (D) at (1,0) ;
\coordinate [label=above:$A_5$,punt] (E) at (1.2,1.8) ;

\draw (A) -- (C) -- (E) -- (B) -- (D) -- (A);

\coordinate (AE) at ($(A)+2*($(B)-(D)$)$) ;
\coordinate (AB) at ($(A)+2*($(E)-(C)$)$) ;
\coordinate (BA) at ($(B)+2*($(C)-(E)$)$) ;
\coordinate (BC) at ($(B)+2*($(A)-(D)$)$) ;
\coordinate (CB) at ($(C)+2*($(D)-(A)$)$) ;
\coordinate (CD) at ($(C)+2*($(B)-(E)$)$) ;
\coordinate (DC) at ($(D)+2*($(E)-(B)$)$) ;
\coordinate (DE) at ($(D)+2*($(C)-(A)$)$) ;
\coordinate (ED) at ($(E)+2*($(A)-(C)$)$) ;
\coordinate (EA) at ($(E)+2*($(D)-(B)$)$) ;


\coordinate [label=above right:$E_1$,punk] (a) at (intersection of B--BC  and E--ED);
\coordinate [label=left:$E_2$,punk] (b) at (intersection of C--CD  and A--AE);
\coordinate [label=right:$E_3$,punk] (c) at (intersection of D--DE  and B--BA);
\coordinate [label=above left:$E_4$,punk] (d) at (intersection of E--EA  and C--CB);
\coordinate [label=below:$E_5$,punk] (e) at (intersection of A--AB  and D--DC);

\draw[thin,green] (A)--(e)--(D)--(c)--(B)--(a)--(E)--(d)--(C)--(b)--(A);

\draw[blue] (A)--(a) (B)--(b) (C)--(c) (D)--(d) (E)--(e);
\coordinate [label=above:\raisebox{6pt}{$M$},punb] (M) at (intersection of  A--a and B--b);

\end{tikzpicture}
\endcomment
\caption{}\label{affinefiveaxes}
\end{figure}

We use homogeneous coordinates and take
$A_1=(0:0:1)$, $A_3=(0:1:1)$, $A_4=(1:0:1)$, $A_2=(a:b:c)$
and $A_5=(x:y:z)$.

The point $E_1$ is easily seen to be $(cx:bz:cz)$.
We compute $E_4=(bx+(c-a)y+(a-c)z:bz:bz)$ and find $M$ as 
the intersection of the axes $g_1=\langle A_1, E_1\rangle$ and
$g_4=\langle A_4, E_4\rangle$.
The result is
\[
M=(cx:bz:(c-b)x+(a-c)y+(c-a+b)z)\;.
\]
In particular, 
$M$ is at infinity if and only if $ (c-b)x+(a-c)y+(c-a+b)z=0$, 
which is the equation of a line whose slope is 
 $\frac{c-b}{c-a}$.
\end{proof}

\begin{remark}\label{computationproof}
With a little more effort one can compute all points $E_i$ and check
that $M$ lies on all axes $g_i=\lijn(A_i,E_i)$. This gives a computational
proof of the five-axes theorem.
\end{remark}
Our stipulation that the conditions of the five-axes theorem be satisfied 
was sufficient for defining the five axes.  
But the resulting
formula for $M$ makes sense under more general circumstances,
indicating that the theorem also holds  in
degenerate cases with a suitable definition of the axes.
The point $M$ 
 fails to be determined only if $cx=bz=-bx+(a-c)(y-z)=0$. When
$A_2$ and $A_5$ are finite points ($c\neq0$ and $z\neq0$), this happens 
if either $A_2=A_4$ and $A_5\in \langle A_1,A_3\rangle$ or
$A_5=A_3$ and $A_5\in \langle A_1,A_4\rangle$.
If, say, $A_5$ lies at infinity ($z=0$), then $A_5=\langle A_1,A_3\rangle
\cap \langle A_2,A_4\rangle$. 
Note that our coordinates
are based on 
the assumption that $A_1$, $A_3$, $A_4$ form a triangle.
In general we can say that the center is undefined 
when for some $i$, $A_{i-1}$
coincides with $A_{i+1}$ and the remaining three points are collinear,
or when $\langle A_{i-1}, A_{i-3}
\rangle \parallel \langle A_{i+1},A_{i+3} \rangle$ with $A_i$ being their
intersection point at infinity,
or when all five points are collinear. 
Moreover, if $M$ is defined, but coincides with the point $A_i$, then the
axis $g_i$ is not defined.

We focus now on one degenerate case, which we need later, 
in which three consecutive
points are collinear: $A_i\in l_i = \langle A_{i-1},A_{i+1}\rangle$.
We have that $ A_i= l_{i-1}\cap l_{i+1}=
l_{i-1}\cap l_i\cap  l_{i+1}$, so $A_i=B_{i-1,i}=B_{i,i+1}$.
In this case the axis $g_i$ can be defined as in Remark
\ref{axis_for-A-in-l}.


\begin{theorem}\label{degen}
Let five points $A_1$, $A_2$, $A_3$, $A_4$ and $A_5$
in the affine plane be given such that $A_5\in \langle A_1,A_4\rangle$, 
but no other three points are collinear. 
Assume that $l_i = \lijn(A_{i-1},A_{i+1})$ is not parallel to 
$l_{i+1} = \lijn(A_i,A_{i+2})$.  
Then  the five axes $g_1$, $g_2$, $g_3$, $g_4$ and $g_5$ 
lie in a pencil.
\end{theorem}

\begin{figure}
\comment
\begin{tikzpicture}
[thick,scale=3,punt/.style={circle,fill=red,inner sep=1pt},
punk/.style={circle,fill=green,inner sep=1pt},
punb/.style={circle,fill=blue,inner sep=1pt}]
\coordinate [label=below right:$A_4$,punt] (A) at (2,0) ;
\coordinate [label=above left:$A_3$,punt] (B) at (0.4,0.6) ;
\coordinate [label=above:$A_2$,punt] (C) at (1.8,1) ;
\coordinate [label=below left:$A_1$,punt] (D) at (0,0) ;
\coordinate [label=below right:$A_5$,punt] (E) at (0.8,0) ;

\draw (A) -- (C) -- (E) -- (B) -- (D) -- (A);

\coordinate (AE) at ($(A)+2*($(B)-(D)$)$) ;
\coordinate (AB) at ($(A)+2*($(E)-(C)$)$) ;
\coordinate (BA) at ($(B)+2*($(C)-(E)$)$) ;
\coordinate (BC) at ($(B)+2*($(A)-(D)$)$) ;
\coordinate (CB) at ($(C)+2*($(D)-(A)$)$) ;
\coordinate (CD) at ($(C)+2*($(B)-(E)$)$) ;
\coordinate (DC) at ($(D)+2*($(E)-(B)$)$) ;
\coordinate (DE) at ($(D)+2*($(C)-(A)$)$) ;
\coordinate (ED) at ($(E)+2*($(A)-(C)$)$) ;
\coordinate (EA) at ($(E)+2*($(D)-(B)$)$) ;


\coordinate [label=above left:$E_4$,punk] (a) at (intersection of B--BC  and E--ED);
\coordinate [label=right:$E_3$,punk] (b) at (intersection of C--CD  and A--AE);
\coordinate [label=left:$E_2$,punk] (c) at (intersection of D--DE  and B--BA);
\coordinate [label=above:$E_1$,punk] (d) at (intersection of E--EA  and C--CB);
\coordinate [label=below:$E_5$,punk] (e) at (intersection of A--AB  and D--DC);

\draw[thin,green] (A)--(e)--(D)--(c)--(B)--(a)--(E)--(d)--(C)--(b)--(A);

\draw[blue] (A)--(a) (B)--(b) (C)--(c) (D)--(d) ;
\coordinate [label=above:$M$,punb] (M) at (intersection of  A--a and B--b);
\draw[blue] (M)--(E)--(e);

\end{tikzpicture}
\endcomment
\caption{$A_5\in l_5$}\label{degen-fig}
\end{figure}
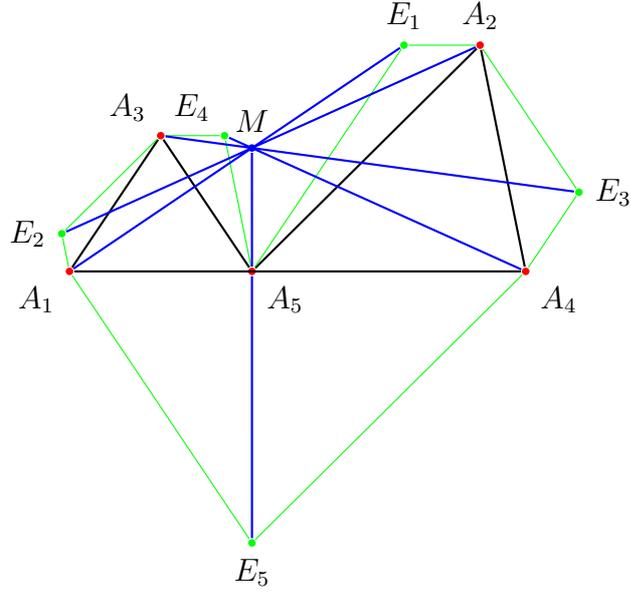

The computation, alluded to in remark \ref{computationproof},
also covers this degenerate case, illustrated in figure \ref{degen-fig}. 
The geometric proof of
the 5-circle theorem in \cite{FHS1, FHS2}   can be extended 
to this situation to show that the four axes $g_1$, $g_2$, $g_3$ and $g_4$
lie in a pencil. If $g_i$ is considered as radical axis of the circles
$c_{i-1,i}$ and $c_{i,i+1}$, this suffices to conclude that all five
radical axes lie in a pencil: 
if the center $M$ is a finite point,
then the fact that $M$ lies on $g_1$, $g_2$, $g_3$ and $g_4$ 
implies that the power of $M$ with respect to $c_{5,1}$ is equal to the
power with respect to $c_{1,2}$, equal to the power with respect to $c_{2,3}$,
$c_{3,4}$ and $c_{4,5}$. As the power of $M$ w.r.t.~$c_{4,5}$ is equal to that
that w.r.t.~$c_{5,1}$, the point $M$ lies on the radical axis $g_5$.
If the center $M$ is infinite, then the centers of all circles involved are
collinear.

The main ingredient of the geometric proof
is Lemma 2 of \cite{FHS1, FHS2}, which we now recall.
\begin{lemma}\label{conditie}
Let $A$, $C$ and $E$ be three non collinear points in $\A^2(k)$, and let
$A$, $C$, $F$, $G$ be collinear, just as $C$, $E$, $H$, $I$ and $B$, $D$,
$G$, $H$ \textup(see figure \textup{\ref{FHS-fig})}. Let $U=[A,F\mid C,G]$,   
$V=[H,D\mid B,G]$ and
$W=[C,H\mid E,I]$. Then the lines $\langle B,U\rangle$, $\langle C,V\rangle$
and $\langle D,W\rangle$ lie in a pencil if and only if
\begin{equation}\label{concur_cond}
\frac{B-G}{B-H}\frac{E-H}{E-C}\frac{F-C}{F-G}=
\frac{D-H}{D-G}\frac{A-G}{A-C}\frac{I-C}{I-H}\;.
\end{equation}
\end{lemma}
\begin{figure}
\comment
\begin{tikzpicture}
[thick,punt/.style={circle,fill,inner sep=1pt},
punk/.style={circle,fill=red,inner sep=1pt}]
\coordinate [label=below left:$A$,punt] (A) at (3.5,-0.5) ;
\coordinate [label=above right:$B$,punt] (B) at (4.5,2) ;
\coordinate [label=above:$C$,punt] (C) at (1,5) ;
\coordinate [label=left:$D$,punt] (D) at (-1.5,2) ;
\coordinate [label=below left:$E$,punt] (E) at (0,0) ;

\coordinate (P) at  (intersection of  E--B and A--D);
\coordinate (Q) at  (intersection of  C--P and A--E);

\coordinate [label=right:$I$,punt](I) at  (intersection of  C--E and D--Q);
\coordinate [label=left:$F$,punt](F) at  (intersection of  C--A and B--Q);

\draw (E) -- (C) -- (A)  (F)-- (B) -- (D) -- (I);

\coordinate [label=above right:$G$,punt](G) at (intersection of  C--A and D--B);
\coordinate [label=above left:$H$,punt](H) at (intersection of  C--E and D--B);

\coordinate (GD) at ($(G)+2*($(D)-(C)$)$) ;
\coordinate (HB) at ($(H)+2*($(B)-(C)$)$) ;
\coordinate (c) at (intersection of G--GD and H--HB);

\coordinate (GA) at ($(G)+2*($(A)-(B)$)$) ;
\coordinate (FC) at ($(F)+2*($(C)-(B)$)$) ;
\coordinate (HE) at ($(H)+2*($(E)-(D)$)$) ;
\coordinate (IC) at ($(I)+2*($(C)-(D)$)$) ;
\coordinate (b) at (intersection of G--GA and F--FC);
\coordinate (d) at (intersection of H--HE and I--IC);

\coordinate [label=below left:$V$,punk](V) at (intersection of C--c and G--H);
\coordinate [label=right:\raisebox{-20pt}{$U$},punk](U) at (intersection of B--b and G--A);
\coordinate [label=below left:$W$,punk](W) at (intersection of D--d and I--H);

\coordinate [label=right:$Z$,punk](Z) at (intersection of D--W and C--B);
\coordinate [label=left:$Y$,punk](Y) at (intersection of C--D and B--U);

\draw [blue] (C)--($(C)!1.1!(V)$) (D)--(Z) (B)--(Y);
\draw [red] (Y)--(C)--(Z);

\end{tikzpicture}
\endcomment
\caption{}\label{FHS-fig}
\end{figure}

\begin{lemma}\label{degencondlemma}
The above lemma also holds if $A$ and $F$ coincide 
\textup(see figure \textup{\ref{degencondfig})}.
\end{lemma}

\begin{proof}
\begin{figure}
\comment
\begin{tikzpicture}
[thick,punt/.style={circle,fill,inner sep=1pt},
punk/.style={circle,fill=red,inner sep=1pt}]
\coordinate [label=below:${A=F}$,punt] (A) at (3.2,0.3) ;
\coordinate [label=above right:$B$,punt] (B) at (4.5,2) ;
\coordinate [label=above:$C$,punt] (C) at (1,5) ;
\coordinate [label=left:$D$,punt] (D) at (-1.5,2) ;
\coordinate [label=below left:$E$,punt] (E) at (0,0) ;

\coordinate (P) at  (intersection of  E--B and A--D);
\coordinate (Q) at  (intersection of  C--P and A--B);

\coordinate [label=right:$I$,punt](I) at  (intersection of  C--E and D--Q);

\draw (E) -- (C) -- (A) -- (B) -- (D) -- (I);

\coordinate [label=above right:$G$,punt](G) at (intersection of  C--A and D--B);
\coordinate [label=above left:$H$,punt](H) at (intersection of  C--E and D--B);

\coordinate (GD) at ($(G)+2*($(D)-(C)$)$) ;
\coordinate (HB) at ($(H)+2*($(B)-(C)$)$) ;
\coordinate (c) at (intersection of G--GD and H--HB);

\coordinate (GA) at ($(G)+2*($(A)-(B)$)$) ;
\coordinate (AC) at ($(A)+2*($(C)-(B)$)$) ;
\coordinate (HE) at ($(H)+2*($(E)-(D)$)$) ;
\coordinate (IC) at ($(I)+2*($(C)-(D)$)$) ;
\coordinate (b) at (intersection of G--GA and A--AC);
\coordinate (d) at (intersection of H--HE and I--IC);

\coordinate [label=below left:$V$,punk](V) at (intersection of C--c and G--H);
\coordinate [label=right:\raisebox{-20pt}{$U$},punk](U) at (intersection of B--b and G--A);
\coordinate [label=below left:$W$,punk](W) at (intersection of D--d and I--H);

\coordinate [label=right:$Z$,punk](Z) at (intersection of D--W and C--B);
\coordinate [label=left:$Y$,punk](Y) at (intersection of C--D and B--U);

\draw [blue] (C)--($(C)!1.1!(V)$) (D)--(Z) (B)--(Y);
\draw [red] (Y)--(C)--(Z);

\end{tikzpicture}
\endcomment
\caption{}\label{degencondfig}
\end{figure}
The proof follows \cite{FHS1, FHS2}. Let $Y=\lijn(B,U) \cap
\lijn(C,D)$ and $Z=\lijn(D,W)\cap \lijn(C,B)$. By Ceva's theorem, applied
to $\triangle BCD$ and its cevians $\lijn (B,Y)$, $\lijn(C,V)$ and 
$\lijn (D,Z)$, the lines $\lijn (B,U)$, $\lijn(C,V)$ and 
$\lijn (D,W)$ lie in a pencil if and only if 
\[
\frac{Y-C}{Y-D}\frac {V-D}{V-B}\frac {Z-B}{Z-C}=-1\;.
\] 
Menelaus' theorem first for $\triangle CDG$ and the points $B$, $U$ and $Y$
and then for $\triangle CBH$ and the points $D$, $W$ and $Z$ gives 
\[
\frac{Y-C}{Y-D}= \frac {U-C}{U-G}\frac {B-G}{B-D}
\quad\text{and}\quad
\frac {Z-B}{Z-C}=\frac {D-B}{D-H}\frac {W-H}{W-C}\;.
\]
The condition $W=[C,H\mid E,I]$ gives by lemma \ref{intsecptlemma} that 
$\frac {W-C}{W-H} = \frac {E-C}{E-H}\frac {I-C}{I-H}$,
while $V=[H,D|B,G]$ gives $\frac {V-D}{V-B}=\frac{D-G}{B-H}$ and 
finally $U=[C,G|A,A]$ implies $\frac{U-G}{U-C}=(\frac{A-G}{A-C})^2$.
Plugging these expression in  in the equation and rearranging  gives that 
$\lijn (B,U)$, $\lijn(C,V)$ and 
$\lijn (D,W)$ lie in a pencil if and only if 
\[
\frac{A-C}{A-G}\frac {B-G}{B-H}\frac {E-H}{E-C}
=\frac{A-G}{A-C}\frac {I-C}{I-H}\frac{D-H}{D-G}\;.
\]
\end{proof}

\begin{proof}
[Proof that $g_1$, \dots, $g_4$ lie in a pencil]
In order to show that the lines $g_1$, $g_2$ and $g_3$ lie in a pencil, we  
verify the condition of lemma \ref{degencondlemma}
with $(B,C,D,E,A=F)=(A_1,A_2,A_3,A_4,A_5=B_{5,1})$,
$(G,H,I,U,V,W)=(B_{1,2},B_{2,3},B_{3,4},C_1,C_2,C_3)$, where $C_i=l_i\cap g_i$.
Both sides of the equation are equal to 1
by Menelaus' theorem applied
to $\triangle CGH$, on the left with the collinear points $B$, $A$ and  $E$, 
and on the right with $D$, $I$ and $A$.
Similarly one shows that $g_2$, $g_3$ and $g_4$ lie in a pencil.
\end{proof}

\section{Six points}
For six points the axes in general do not lie in a pencil. 

\begin{theorem}\label{6axis}
Let six points $A_1$, \dots, $A_6$, be given,  no three collinear
and such that the six points $B_{i,i+1}=\langle A_{i-1},A_{i+1}\rangle
\cap \langle A_i,A_{i+2}\rangle$ are finite.
Then the following are equivalent:
\begin{enumerate}
\item the six axes $g_1$, \dots, $g_6$, lie in a pencil
\item for some $i$ the axes $g_{i-1}$, $g_i$, $g_{i+1}$ lie in a pencil,
\item the main diagonals of the hexagon
$A_1A_2A_3A_4A_5A_6$ lie in a pencil,
\item the six points $B_{i,i+1}$ lie on a conic.
\end{enumerate}
\end{theorem}

\begin{figure}
\comment
\begin{tikzpicture}
[thick,punt/.style={circle,fill,inner sep=1.5pt}, punk/.style={circle,draw,minimum size=4pt, inner sep=0pt}, pung/.style={circle,fill=green,inner sep=1.5pt}]

\node(e) [ellipse,red,minimum width=7cm,draw] at (0,0) {\vrule width 0pt height 3cm} ;

\coordinate [punt,label=above right:$B_{1,2}$] (b12) at (e.30) ;
\coordinate [punt,label=above:$B_{2,3}$] (b23) at (e.90) ;
\coordinate [punt,label=above left:$B_{3,4}$] (b34) at (e.150) ;
\coordinate [punt,label=left:$B_{4,5}$] (b45) at (e.-165) ;
\coordinate [punt,label=below:$B_{5,6}$] (b56) at (e.-90)  ;
\coordinate [punt,label=below right:$B_{6,1}$] (b61) at (e.-15)  ;

\coordinate [punt,label=right:$A_{1}$] (a1) at (intersection of  b23--b12 and b56--b61);
\coordinate [punt,label=right:$A_{2}$] (a2) at (intersection of  b34--b23 and b61--b12);
\coordinate [punt,label=left:$A_{3}$] (a3) at (intersection of  b45--b34 and b12--b23);
\coordinate [punt,label=left:$A_{4}$] (a4) at (intersection of  b56--b45 and b23--b34);
\coordinate [punt,label=left:$A_{5}$] (a5) at (intersection of  b61--b56 and b34--b45);
\coordinate [punt,label=right:$A_{6}$] (a6) at (intersection of  b12--b61 and b45--b56);

\coordinate [punk]  (e1) at (intersection of b12--{$(b12)+(a6)-(a1)$} and b61--{$(b61)+(a2)-(a1)$});
\coordinate [punk]  (e2) at (intersection of b23--{$(b23)+(a1)-(a2)$} and b12--{$(b12)+(a3)-(a2)$});
\coordinate [punk]  (e3) at (intersection of b34--{$(b34)+(a2)-(a3)$} and b23--{$(b23)+(a4)-(a3)$});
\coordinate [punk]  (e4) at (intersection of b45--{$(b45)+(a3)-(a4)$} and b34--{$(b34)+(a5)-(a4)$});
\coordinate [punk]  (e5) at (intersection of b56--{$(b56)+(a4)-(a5)$} and b45--{$(b45)+(a6)-(a5)$});
\coordinate [punk]  (e6) at (intersection of b61--{$(b61)+(a5)-(a6)$} and b56--{$(b56)+(a1)-(a6)$});
\coordinate [punk] (n) at (intersection of a1--a4 and a2--a5);
\draw [thin,green] (e1)--(b12)--(e2)--(b23)--(e3)--(b34)--(e4)--(b45)--(e5)--(b56)--(e6)--(b61)--(e1);
\coordinate [punt] (m) at (intersection of a1--e1 and a2--e2);

\coordinate (d12) at (intersection of a3--a4 and a5--a6);
\coordinate (d23) at (intersection of a4--a5 and a6--a1);
\coordinate (d34) at (intersection of a5--a6 and a1--a2);
\coordinate (d45) at (intersection of a6--a1 and a2--a3);
\coordinate (d56) at (intersection of a1--a2 and a3--a4);
\coordinate (d61) at (intersection of a2--a3 and a4--a5);

\coordinate [pung] (t12) at (intersection of a1--a2 and d12--b12);
\coordinate [pung] (t23) at (intersection of a2--a3 and d23--b23);
\coordinate [pung] (t34) at (intersection of a3--a4 and d34--b34);
\coordinate [pung] (t45) at (intersection of a4--a5 and d45--b45);
\coordinate [pung] (t56) at (intersection of a5--a6 and d56--b56);
\coordinate [pung] (t61) at (intersection of a6--a1 and d61--b61);

\draw [green] (t12) .. controls ($(t12)!0.57!(a2)$) and ($(t23)!0.57!(a2)$)  .. (t23)  .. controls ($(t23)!0.57!(a3)$) and ($(t34)!0.57!(a3)$)  .. 
(t34)  .. controls ($(t34)!0.57!(a4)$) and ($(t45)!0.57!(a4)$) .. 
(t45) .. controls ($(t45)!0.57!(a5)$) and ($(t56)!0.57!(a5)$) .. 
(t56) .. controls ($(t56)!0.57!(a6)$) and ($(t61)!0.57!(a6)$) ..
(t61) .. controls ($(t61)!0.57!(a1)$) and ($(t12)!0.57!(a1)$) ..(t12) ; 

\draw [thin,brown] (a1) -- (a2) -- (a3) -- (a4) -- (a5) -- (a6) -- (a1);
\draw (a1) -- (a3) --(a5) -- (a1) (a2) -- (a4) -- (a6) -- (a2);

\draw [blue] (a1)--(e1)--(m)--(e5)--(a5) (a6)--(e6)--(m);
\draw [blue] (a4)--(e4)  --  node[below left] {$g_4$} (m) ;
\draw [blue] (a2)--(e2)  --  node[above right] {$\;\;g_2$} (m) ;
\draw [blue] (a3)--(e3)  --  node[left] {$g_3$} (m) ;
\draw [dashed] (a1)--(n)--(a4) (a2)--(n)--(a5) (a3)--(n)--(a6);

\end{tikzpicture}
\endcomment
\caption{}\label{ellipse}
\end{figure}

\begin{proof}
$(1)\implies (2) \implies (3) \implies (1)$:\\
We  show that condition (3) is equivalent to $g_{i-1}$, $g_i$, $g_{i+1}$ 
lying in a pencil for all $i$. But as
the condition on the main diagonals does not single out three
lines, it suffices to prove equivalence for one specific $i$, say $i=3$.

We take affine coordinates $(x,y)$ 
with $A_3$ as origin, $A_2=(0,1)$, $A_4=(1,0)$,
$A_1=(a,b)$, $A_6=(c,d)$ and $A_5=(e,f)$.
We compute $B_{2,3}=(\frac a{a+b},\frac b{a+b})$, 
$B_{3,4}=(\frac e{e+f},\frac f{e+f})$,
$B_{1,2}=(\frac {ac}{bc-ad+a},\frac {bc}{bc-ad+a})$
and
$B_{4,5}=(\frac {ed}{ed-fc+f},\frac {fd}{ed-fc+f})$.
The condition  \eqref{concur_cond} of lemma \ref{conditie}
(with the labels $A, \dots, I$ applied, 
in order, to $A_1, \dots, A_5$, $B_{1,2}$,
$B_{2,3}$, $B_{3,4}$, $B_{4,5}$) %
then becomes
\[
\frac {a(e+f)}{(a+b)e} \cdot \frac{e+f-1}{e+f} \cdot \frac{c(a+b)}{(c+d-1)a}
=
\frac{f(a+b)}{(e+f)b}\cdot\frac{a+b-1}{a+b}\cdot\frac{d(e+f)}{f(c+d-1)}\;,
\]
which simplifies to
\begin{equation}\label{zescond}
(e+f-1)cb=(a+b-1)de\;.
\end{equation}
Here we used $a+b\neq 0$ (as $l_2\nparallel l_3$), $e+f\neq0$ and $a\neq0$
(as $A_2\notin l_2$), $f\neq0$ and $c+d\neq1$ (as $A_6\notin l_3$).

The diagonal $\lijn(A_3,A_6)$ has equation $dx-cy=0$, 
the diagonal $\lijn(A_1,A_4)$
is given by $bx+(1-a)y=b$ and  $\lijn(A_2,A_5)$ by $(1-f)x+ey=e$.
The condition that these three diagonals lie in a pencil is given
by the vanishing of the determinant
\[
\Delta= 
\begin{vmatrix}
b & 1-a &-b\\
1-f&  e &-e\\
d &  -c & 0
\end{vmatrix} = 
\begin{vmatrix}
0 & 1-a-b &-b\\
1-e-f&  0 &-e\\
d &  -c & 0
\end{vmatrix}
\;.
\]
Computing this determinant with Sarrus' rule shows that
$\Delta=0$ if and only if equation \eqref{zescond} holds.

$(3) \iff  (4)$:\\
The lines $\langle A_1, A_4\rangle$,  $\langle A_2, A_5\rangle$
and  $\langle A_3, A_6\rangle$ lie in a pencil if and only if the triangles 
$\triangle A_1A_3A_5$ and $\triangle A_4A_6A_2$ are perspective 
 from a center which, by Desargues's theorem, holds if and only if they are perspective from an axis.  
Note that the line $l_i=\langle A_{i-1},A_{i+1}\rangle$ coincides with the line
$\langle B_{i-1,i}, B_{i,i+1} \rangle$. Therefore the axis of
perspectivity is also the Pascal line of the points $B_{i,i+1}$,
whence these points lie on a conic if and only if the original three 
lines lie in a pencil.
\end{proof}

\begin{remark}
If $\operatorname{char} k\neq 2$ the hexagon $A_1A_2A_3A_4A_5A_6$
circumscribes a conic by Brianchon's theorem. This is not true
in characteristic 2, as then all tangents to a conic pass through
one point.
Figure \ref{ellipse} illustrates the result in the 
euclidean plane. To make the conics clearly visible the axes $g_i$ 
are constructed  by drawing parallels through
$B_{i-1,i}$ and $B_{i,i+1}$.
\end{remark}

\begin{remark}\label{zesdegen}
The above proof shows that under weaker conditions, 
the equivalence between the axes $g_2$,
$g_3$ and $g_4$ lying in a pencil and the main diagonals lying in a 
pencil continues to hold.
The condition  \eqref{concur_cond}
applied to $(A,B,C,D,E)=(A_1,A_2,A_3,A_4,A_5)$ does not involve
the position of the point $A_6$. The proof, 
when written in homogeneous coordinates, therefore remains valid
should $A_6$ lie at infinity ($l_1\parallel l_5$), 
or should $A_6\in l_2,l_4,l_6$.
Also the degenerations $A_1\in l_5$, $A_5\in l_1$, $A_2\in l_6$, $A_4\in l_6$ or
$l_5\parallel l_6$, $l_1\parallel l_6$ do not affect the conclusion.
\end{remark}

\section{The proof of the main result}
We have now seen that Theorem \ref{mainthm} holds for extended versions
of the cases $n=5$ and $n=6$.  For $n \ge 7$ we find it convenient to assume 
that the axes $g_2$, \dots, $g_{n-2}$ lie in a pencil.

The proof of Theorem \ref{mainthm} proceeds by induction on the number of vertices.
The idea is the following. Suppose $A_1$, \dots, $A_n$ are given 
with $g_2$, \dots, $g_{n-2}$ in a pencil. 
Then we construct a sequence $A_1$, $A_2$, $A_{3,4}$,
$A_5$, \dots, $A_n$ of $n-1$ points by replacing $A_3$ and $A_4$ 
by the intersection $A_{3,4}$ of $l_2$ and $l_5$. For the new configuration
the axes $g_2$, $g_ {3,4}$,  $g_5$, \dots, $g_{n-2}$ lie in a pencil
with the same center, and the induction hypothesis applies, provided
the configuration satisfies the assumptions of the theorem.
Sometimes this will not be the case, but we shall see that without loss of 
generality,  one can replace the given configuration by one which does 
satisfy the assumptions.

Three consecutive axes $g_{i-1}$, $g_i$, $g_{i+1}$
are determined by seven points $A_{i-3}$, $A_{i-2}$, $A_{i-1}$, $A_i$,
$A_{i+1}$, $A_{i+2}$ and $A_{i+3}$. Let $D_i$ be the (possibly infinite) 
intersection point of $l_{i-2}=\lijn(A_{i-3},A_{i-1})$ and  
$l_{i+2}=\lijn(A_{i+1},A_{i+3})$. The point $D_i$ exists
as $l_{i-2}\neq l_{i+2}$, because $A_{i+1}\notin l_{i-2}$. The axes
$g_{i-1}$, $g_i$, $g_{i+1}$ are also the axes through $A_{i-1}$, 
$A_i$, $A_{i+1}$ in the hexagon $D_iA_{i-2}A_{i-1}A_iA_{i+1}A_{i+2}$.
This hexagon does not necessarily satisfy all the conditions \eqref{ass1},
\eqref{ass2}, \eqref{ass3}, but by remark \ref{zesdegen} less is
needed to conclude that the lines $g_{i-1}$, $g_i$, $g_{i+1}$
lie in a pencil if and only if the lines $\lijn(A_{i-2},A_{i+1})$,
$\lijn(A_{i-1},A_{i+2})$ and $\lijn(A_i,D_i)$ lie in a pencil
(see also figure \ref{zescriterium}). Only the three conditions
$l_{i+1}\neq \lijn(A_{i+2},A_{i-2})$, $l_{i-1}\neq \lijn(A_{i+2},A_{i-2})$
and $l_{i-2}\neq l_{i+2}$ are not directly covered by the properties
of the original configuration and the allowable degenerations from the remark.
We already showed that $l_{i-2}\neq l_{i+2}$. If 
$l_{i-1}= \lijn(A_{i+2},A_{i-2})$, then $A_{i-2}$, $A_i$ and $A_{i+2}$
are collinear, which whould imply that $l_{i-1}=l_{i+1}$,
contradicting the condition  \eqref{ass2} for the original 
configuration; for the same reason  $l_{i+1}\neq \lijn(A_{i+2},A_{i-2})$.
So the condition to test is indeed that each triple
of lines $\lijn(A_{i-2},A_{i+1})$,
$\lijn(A_{i-1},A_{i+2})$ and $\lijn(A_i,D_i)$ lies in a pencil.

\begin{figure}
\comment
\begin{tikzpicture}
[thick,punt/.style={circle,fill=red,inner sep=1pt},
punk/.style={circle,fill=green,inner sep=1pt}]

\coordinate [label=left:$A_{i-2}$,punt] (B) at (0,5) ;
\coordinate [label=left:$A_{i-1}$,punt] (C) at (-2,2.5) ;
\coordinate [label=below left:$A_i$,punt] (D) at (0,0) ;
\coordinate [label=below right:$A_{i+1}$,punt] (E) at (5,1) ;
\coordinate [label=right:$A_{i+2}$,punt] (F) at (7,2.5) ;
\coordinate [label=right:$A_{i+3}$,punt] (G) at (5,3.5) ;

\coordinate (M) at (intersection of B--E and C--F);
\coordinate  [label=above:$D_i$,punk] (Q) at (intersection of D--M and E--G);

\coordinate [label=above left:$A_{i-3}$,punt] (A) at ($(C)!0.7!(Q)$) ;

\draw (B) -- (D) -- (F)  (G) -- (E) -- (C)--(A);
\draw  [dashed] (G)--(Q)--(A);

\draw[thin,brown] (B)--(E) (D)--(Q) (C)--(F) ;

\end{tikzpicture}
\endcomment
\caption{}\label{zescriterium}
\end{figure}

In the following lemma we consider a sequence of points $A_0$, $A_1$, \dots,
$A_6$, which may be part of a larger configuration. Because of the lemma's 
limited scope, 
we require only that the indices in the assumptions
\eqref{ass1} -- \eqref{ass3} 
lie between 0 and 6.

\begin{lemma}\label{move}
Let $A_0$, $A_1$, \dots, $A_5$, $A_6$ be a sequence of 
distinct points satisfying 
the assumptions \eqref{ass1} -- \eqref{ass3} 
limited to indices between 0 and 6, 
such that the axes $g_2$, $g_3$
and $g_4$ lie in a pencil.
Choose a point $A_3'\in l_4$ with $A_3'\neq
B_{3,4}$ and  $A_3'\neq \langle A_1,A_2\rangle\cap l_4$.
Let $P=\langle A_1,A_4\rangle \cap \langle A_2,A_3'\rangle $.
Define the point $A_2'\in l_1$ as $A_2'=l_1\cap \langle P, A_3\rangle$.
Suppose that the
sequence $A_0$, $A_1$, $A_2'$, $A_3'$, $A_4$, $A_5$, $A_6$
also satisfies the 
limited
assumptions \eqref{ass1} -- \eqref{ass3}. Denote the axes
of this new configuration by $g_i'$. Then $g_4=g_4'$ and the
the axes $g_2'$, $g_3'$ and $g_4'$ lie in the same pencil
as $g_2$, $g_3$ and $g_4$.
If moreover one of the axes $g_1$, $g_1'$  is defined and also lies in the 
same pencil,
then $g_1'=g_1$.
\end{lemma}

\begin{figure}
\comment
\begin{tikzpicture}
[thick,punt/.style={circle,fill=red,inner sep=1pt},
punk/.style={circle,fill=green,inner sep=1pt},
punb/.style={circle,fill,inner sep=1pt}]

\coordinate [label=left:$A_{1}$,punt] (B) at (0,5) ;
\coordinate [label=left:$A_{2}$,punt] (C) at (-1,1.5) ;
\coordinate [label=below:$A_3$,punt] (D) at (0,0) ;
\coordinate [label=below right:$A_{4}$,punt] (E) at (5,1) ;
\coordinate [label=right:$A_{5}$,punt] (F) at (7,2.5) ;
\coordinate [label=right:$A_{6}$,punt] (G) at (5,3.5) ;

\coordinate [label=below:$A_3'$,punt] (DD) at ($(F)!1.1!(D)$) ;

\coordinate (N) at (intersection of B--E and C--F);
\coordinate   (Q) at (intersection of D--N and E--G);

\coordinate [label=below right:$A_{0}$,punt] (A) at ($(C)!0.8!(Q)$) ;

\coordinate [label=above:$P$,punb] (P) at (intersection of B--E and DD--C);
\coordinate [label=above:$A_2'$,punt] (CC) at (intersection of A--C and P--D);

\coordinate (BC) at (intersection of A--C and B--D) ;
\coordinate (CD) at (intersection of B--D and C--E) ;
\coordinate  [label=below:$B_{3,4}$,punb] (DE) at (intersection of C--E and D--F) ;
\coordinate [punb] (EF) at (intersection of D--F and E--G) ;
\coordinate (BCC) at (intersection of A--CC and B--DD) ;
\coordinate (CCDD) at (intersection of B--DD and CC--E) ;
\coordinate  [punb](DDE) at (intersection of CC--E and DD--F) ;

\coordinate (c) at (intersection of B--{$(B)+(E)-(C)$} and D--{$(D)+(Q)-(C)$});
\coordinate (cc) at (intersection of B--{$(B)+(E)-(CC)$} and DD--{$(DD)+(Q)-(CC)$});
\coordinate (d) at (intersection of C--{$(C)+(F)-(D)$} and E--{$(E)+(B)-(D)$});
\coordinate (dd) at (intersection of CC--{$(CC)+(F)-(DD)$} and E--{$(E)+(B)-(DD)$});

\coordinate (e) at (intersection of D--{$(D)+(G)-(E)$} and F--{$(F)+(C)-(E)$});
\coordinate (ee) at (intersection of DD--{$(DD)+(G)-(E)$} and F--{$(F)+(CC)-(E)$});
\coordinate (eee) at (intersection of DD--{$(DD)+(C)-(E)$} and D--{$(D)+(CC)-(E)$});

\coordinate [label=below left:$M$,punk] (M) at (intersection of C--c and D--d) ;

\draw (B) -- (D) -- (F) (DD) -- (D)  (G) -- (E) -- (C)--(A);
\draw [dashed] (B)--(DD) (CC)--(E);
\draw[thin,green] (e)--(F);
\draw[thin,green,dashed]  (DD)--(ee)--(F)  ;
\draw[thin,brown] (D)--(eee)--(DD);

\draw[thin,red] (DD)--(P)--(B)--(E) (P)--(CC)--(D);

\draw[blue]   (E)--(M);
\draw[blue,thin]   (e)--(M) (E)--(eee);
\draw[blue,thin,dashed] (e)--(ee) ;


\end{tikzpicture}
\endcomment
\caption{}\label{movefig}
\end{figure}

\begin{proof}
The construction is illustrated in figure \ref{movefig}.
We want  to apply the 6-axes theorem (Theorem \ref{6axis}) to the points $A_1$, $A_2$, $A_3$, $A_4$,
$A_3'$, $A_2'$. Therefore we check that 
they are distinct and
satisfy assumptions \eqref{ass1} -- \eqref{ass3}.

By construction $A_2'=A_2$ if and only if $A_3'=A_3$, 
but then there is nothing to prove. We therefore assume
$A_3'\neq A_3$. This also gives $l_3'\neq l_3$ and $l_2'\neq l_2$. 
We have that $A_3'\in l_4$; as $A_2\notin l_4$ and $A_4\notin l_4$,
$A_3'\neq A_2$ and 
$A_3'\neq A_4$; similarly for $A_2'$.
The only other requirements that do not follow from the assumptions
on $A_0$, $A_1$,  $A_2$, $A_3$, $A_4$, $A_5$, $A_6$ and 
$A_0$, $A_1$, $A_2'$, $A_3'$, $A_4$, $A_5$, $A_6$, are
$A_2\notin l_2'$, $A_2'\notin l_2$, $A_3\notin l_3'$ and $A_3'\notin l_3$.

If $A_3'\in l_3$, then $A_3'=B_{3,4}$. 
If $A_3\in l_3'=\langle A_2',A_4\rangle$, then $A_4\in \langle A_2',A_3 \rangle
\cap \langle A_1,A_4\rangle = \{P\}$, which again implies the 
excluded case 
$B_{3,4} =A_2A_4 \cap A_3A_5 = A_2P \cap A_3A_5 = A_3'$.

The condition $A_3'\neq \langle A_1,A_2\rangle\cap l_4$ gives
$A_2\notin \langle A_1,A_3'\rangle= l_2'$.
If $A_2'\in l_2=\langle A_1,A_3\rangle$, then $P\in \langle A_1,A_3\rangle$.
As also $P\in \langle A_1,A_4\rangle$ this implies that $P=A_1$ and
again $A_2\in \langle A_1,A_3'\rangle=l_2'$.

As $\lijn(A_1,A_4)$, $\lijn(A_2,A_3')$ and $\lijn(A_3,A_2')$ lie
in a pencil, the axes $\bar g_1$, $\bar g_2$, $\bar g_3$,
$\bar g_4$, $\bar g_3'$ and $\bar g_2'$ of the hexagon
$A_1A_2A_3A_4A_3'A_2'$ lie in a pencil. Because  $A_0$, $A_2$ and $A_2'$ are 
collinear and also $A_5$, $A_3$ and $A_3'$, we have that $\bar g_2=g_2$,
$\bar g_3=g_3$,  $\bar g_2'=g_2'$ and $\bar g_3'=g_3'$.

As $g_4$ lies in the pencil of $g_2=\bar g_2$ and  $g_3=\bar g_3$,
the axis $g_4$ also  coincides with $\bar g_4$.
The axis $g_4$ is constructed as $\lijn(A_4,E_4)$, with $E_4$ the
intersection point of the parallel to $l_5=\lijn(A_4,A_6)$  through $A_3$
and the parallel to $l_3=(A_4,A_2)$ through $A_5$. 
For  $\bar g_4=\lijn(A_4,\bar E_4)$ one finds $\bar E_4$ as intersection
point of  the parallel to $\bar l_3=l_3=\lijn(A_4,A_2)$ through $A_3'$
and the parallel to $\bar l_3'=l_3'=\lijn(A_4,A_2')$ through $A_3$.
For $g_4'=\lijn(A_4,E_4')$ one 
intersects the parallel to $l_5=\lijn(A_4,A_6)$ 
through $A_3'$
with the parallel to $l_3'=(A_4,A_2')$ through $A_5$.
By Pappus' theorem applied to the collinear points $A_3$, $A_3'$ and $A_5$
and the points at infinity of the three lines $l_3$, $l_3'$ and $l_5$
the points $E_4$, $\bar E_4$ and $E_4'$ are collinear. As $E_4$ and $\bar E_4$
lie on $g_4=\bar g_4$, the point $E_4'$ also lies on it and therefore
$g_4'=g_4$.

Furthermore, if $g_1$ lies in the  pencil the same argument
gives that $g_1'=g_1$.
\end{proof}

\begin{remark}
The requirement that the
sequence $A_0$, $A_1$, $A_2'$, $A_3'$, $A_4$, $A_5$, $A_6$
also satisfy assumptions \eqref{ass1} -- \eqref{ass3}
implies only finitely many forbidden positions
for $A_3'$. Those can be made explicit.
One finds that $A_3'$ should not be equal to $A_5$, $\lijn(A_0,A_1)\cap l_4$,
$l_1\cap l_4$, $l_5\cap l_4$, $\lijn(A_1,A_4)\cap l_4$ and also not
equal to $l_4\cap m_1$, where $m_1$ is the line through $A_1$, parallel to
$l_1$. 
There will, of course, be a few more forbidden positions when
the given points 
are part of a larger configuration. 
For $A_2'$ one has corresponding forbidden positions. As one finds $A_2'$ from $A_3'$
by first projecting $l_4$ from $A_2$ onto the line $\lijn(A_1,A_4)$ and
then projecting from $A_3$ on the line $l_1$, those positions of $A_2'$ yield
further forbidden positions of $A_3'$.  

This covers all assumptions except
$l_2\nparallel l_3'$, $A_2'\notin l_2'$ and $A_3'\notin l_3'$. They involve the
position of the point $B_{2,3}$: in the first case it lies at infinity, 
in the second $B_{2,3}=A_2'$ and finally $B_{2,3}=A_3'$. 
The point $B_{2,3}'$ is the intersection $\lijn(A_1,A_3')\cap
\lijn(A_4,B_{3,4}')$ and as $(A_3',B_{3,4}')$ is a pair of an involution
on $l_4$ the point $B_{2,3}'$ moves on a (possibly degenerate) conic
through $A_1$, $A_4$ and $B_{2,3}$, as $A_3'$ moves on $l_4$. The intersection
of this conic with the line at infinity, $l_1$ and $l_4$ gives at most
six forbidden positions for $B_{2,3}'$ and therefore for $A_3'$.

On the other hand, because we could, if needed, embed the given plane 
in a plane over a field extension
we can assume without loss of generality that there are infinitely many 
allowable positions for $A'_3$ on $l_4$.
\end{remark}

\begin{proof}[Proof of Theorem \ref{mainthm}]
Suppose distinct points $A_1$, \dots, $A_n$  ($n\geq 7$)
are given, satisfying the assumptions
\eqref{ass1},  \eqref{ass2}, \eqref{ass3} 
and such that the $n-3$ lines
$g_2$, \dots, $g_{n-2}$ lie in a pencil.
The lines $l_2$ and $l_5$ are not equal, as $A_4\in l_5$, but $A_4\notin l_2$.\\
Let
\begin{align*}
A_{3,4}&=l_2 \cap l_5 \text{ (possibly at infinity),}\quad
 l_{3,4} = \lijn(A_2,A_5), \\
 B_3 &= l_{6,4} \cap l_2\;
\text{ and } \;B_4 = l_{3,4} \cap l_5.
\end{align*}
Consider the sequence of $n-1$ points $A_1$, $A_2$, $A_{3,4}$, $A_5$,
\dots, $A_n$.
Suppose first that $A_{3,4}$ is a finite point and that the sequence also
satisfies the assumptions 
\eqref{ass1} -- \eqref{ass3}, as in figure \ref{mainfig}.

\begin{figure}
\comment
\begin{tikzpicture}
[thick,scale=1.2,punt/.style={circle,fill=red,inner sep=1pt},
punk/.style={circle,fill,inner sep=1pt}]

\coordinate [label=below:$A_{3,4}$,punt] (p34) at (0,0)  ;
\coordinate [label=left:$A_3$,punt] (p3) at (-3,2.5);
\coordinate [label=right:$A_4$,punt] (p4) at (3,3);
\coordinate [label=below:$A_6$,punt] (p6) at ($(p34)!0.33!(p4)$) ;
\coordinate [label=below:$A_1$,punt] (p1)  at ($(p34)!0.8!(p3)$) ;
\coordinate [label=right:$A_2$,punt] (p2) at (2,6);
\coordinate [label=left:$A_5$,punt] (p5) at (-2,6);

\coordinate (e3) at (intersection of p2--{$(p2)+(p5)-(p3)$} and p4--{$(p4)+(p3)$});
\coordinate (e4) at (intersection of p3--{$(p3)+(p4)$} and p5--{$(p5)+(p2)-(p4)$});

\coordinate (e34) at (intersection of p2--{$(p2)+(p4)$} and p5--{$(p5)+(p3)$});

\coordinate (e2) at (intersection of p34--{$(p2)-(p4)$} and p3--{$(p3)+(p2)-(p5)$});
\coordinate (e5) at (intersection of p34--{$(p5)-(p3)$} and p4--{$(p4)+(p5)-(p2)$});

\coordinate [label=above:$M$,punt] (M) at (intersection of p3--e3 and p4--e4) ;

\coordinate (E2) at (intersection of p2--e2 and p1--{$(p1)+(p4)-(p2)$});
\coordinate (E5) at (intersection of p5--e5 and p6--{$(p6)+(p3)-(p5)$});

\coordinate (p0) at ($(p2)+0.9*($(p3)-(E2)$)$);
\coordinate (p7) at ($(p5)+1.3*($(p4)-(E5)$)$);

\coordinate (E5p) at (intersection of p6--{$(p6)+(p2)-(p5)$} and p34--{$(p7)-(p5)$});

\coordinate (E2p) at (intersection of p1--{$(p1)+(p5)-(p2)$} and p34--{$(p0)-(p2)$});

\draw (p1)--(p3)--(p5)--(p7) (p0)--(p2)--(p4)--(p6) ;
\draw [dashed] (p6)--(p34)--(p1) (p2)--(p5);

\draw[green,thin,dashed] (p2)--(e34)--(p5);
\draw[green,thin] (p5)--(e4)--(p3)--(E2)--(p1) (p6)--(E5)--(p4)--(e3)--(p2);

\draw[blue] (p3)--(M)--(e3) (e4)--(M)--(p4) (E2)--(M)--(p2) (p5)--(M)--(E5p);
\draw[blue,dashed]  (p34)--(M)--(e34);
\draw[brown,thin] (p3)--(e2)--(p34)--(e5)--(p4) ;
\draw[green,thin,dashed] (p6)--(E5p)--(p34) (p1)--(E2p)--(p34);
\end{tikzpicture}
\endcomment
\caption{}\label{mainfig}
\end{figure}

The lines $l_2$ and $l_5$ occur both in the configuration of $n$ points
and of $n-1$ points, and also in the configuration formed by the five
points  $A_2$, $A_3$, $A_4$, $A_5$, $A_{3,4}$. 
Now $A_{3,4}\neq A_3$, as $A_{3,4}\in l_5$ but $A_3\notin l_5$;
similarly   $A_{3,4}\neq A_4$.

We verify the conditions
\eqref{ass1} -- \eqref{ass3} for the pentagon $A_2A_3A_4A_5A_{3,4}$. 
Most of them are 
conditions which also appear as conditions for 
$A_1$, $A_2$, $A_{3,4}$, $A_5$, \dots, $A_n$ or 
$A_1$, $A_2$, $A_3$, $A_4$, \dots, $A_n$.
For \eqref{ass1} we note that $A_3\notin l_{3,4}$,
as $A_2$, $A_3$ and $A_5$ are not collinear, because $A_2\notin l_4$; 
 likewise 
  $A_4\notin l_{3,4}$.  Also 
$A_{3,4}\notin l_4$, for otherwise $A_{3,4}=l_2\cap l_4=A_3$,
similarly $A_{3,4}\notin l_3$. For
\eqref{ass2} we  have $l_{3,4}\neq l_4$ (and similarly $l_{3,4}\neq l_3$)
because $A_2$, $A_3$ and $A_5$ are not collinear.

Therefore the 5-axes theorem 
applies
to the configuration $A_2$, $A_3$, $A_4$, $A_5$, $A_{3,4}$.
Its axes $\bar g_2$, $\bar g_3$, $\bar g_4$, $\bar g_5$ and $\bar g_{3,4}$
lie in a pencil. As $\bar g_3$ coincides with the axes $g_3$ of 
the configuration $A_1$, $A_2$, $A_3$, $A_4$, \dots, $A_n$, and likewise
$\bar g_4=g_4$, and $g_2$ and $g_5$ lie in a pencil with
$g_3$ and $g_4$, we find that also $\bar g_2=g_2$ and $\bar g_5=g_5$.
By the same argument as in the previous proof we conclude that
$g_5$ is also the axis through $A_5$  in the configuration 
$A_1$, $A_2$, $A_{3,4}$, $A_5$, \dots, $A_n$, and a similar statement
holds for $g_2$. The axes $\bar g_{3,4}$ is also the axis through $A_{3,4}$
in the configuartion of $n-1$ points.
Therefore  the $n-4$ axes $g_2$, $g_{3,4}$, $g_5$, \dots, 
$g_{n-2}$ lie in a pencil and by the induction hypothesis
the axes $g_1$,  $g_{n-1}$ and $g_n$ lie in the same  pencil, which is 
also the pencil of $g_2$, $g_3$, $g_4$, $g_5$, \dots, 
$g_{n-2}$.

If $A_{3,4}$ lies at infinity
or coincides with one of the other points,  or the configuration
$A_1$, $A_2$, $A_{3,4}$, $A_5$, \dots, $A_n$ 
does not satisfy the assumptions \eqref{ass1} -- \eqref{ass3},
we use the construction of lemma \ref{move} to replace 
$A_1$, $A_2$, $A_3$, $A_4$, $A_5$, \dots, $A_n$ by another
one $A_1'$, \dots, $A_n'$ with the same center, such that
$A_1'$, $A_2'$, $A_{3,4}'$, $A_5'$,
\dots, $A_n'$ does satisfy the assumptions. 
As mentioned earlier, the 
 new sequence need not be defined over the field $k$; it 
suffices for the induction that it is defined over a field extension.

Some of the assumptions 
\eqref{ass1} -- \eqref{ass3} 
for the configuration $A_1$, $A_2$, $A_{3,4}$, $A_5$, \dots, $A_n$
follow directly from the properties \eqref{ass1} -- \eqref{ass4}
of the $n$ points
$A_1$,  \dots, $A_n$, but for the others we have to modify the given
configuration. We do this step by step. 
At each step we maintain $n-2$ points from the previous step and move the other
two in a way that corrects one specific shortcoming  (it is here that we might 
have to make use of a field extension). We then relabel the points so that the 
resulting configuration is free of all previous shortcomings, yet has the same 
center.

We now list the conditions and discuss how to satisfy them. 
We treat the cases which are connected by the symmetry $A_n\mapsto
A_{7-n}$ together, postponing $A_{3,4}\notin l_{3,4}$ to the end.
\begin{itemize}
\item $l_2\nparallel l_5$.\\
This 
condition implies that the point $A_{3,4}=l_2\cap l_5$ is a finite point, 

as desired.
As $l_2\neq l_4$, $A_1\notin l_4$. Moving $A_3$ on $l_4$ means that 
the line $l_2$ moves in the pencil of lines through $A_1$, whereas $l_4$
does not change. Therefore, if we were given $l_2\parallel l_5$, we could make these lines
intersecting by moving $A_2$ and $A_3$. 
\item $A_1\neq A_{3,4}$ and  $A_6\neq A_{3,4}$.\\
If  $A_{3,4}=A_6$, then $l_2=\lijn(A_1,A_3)$ intersects $l_5=\lijn(A_4,A_6)$
in $A_6$. Moving $A_3$ on $l_4$ means that $l_2$ moves in the pencil of
lines with center $A_1$. As $A_6\neq A_1$, this means that $A_{3,4}$ 
moves. If $A_{3,4}=A_1$, we move instead $A_4$ on $l_3$.
\item $A_{3,4}\neq A_j$ for $j=7,\dots,n$.\\
If $A_j=l_2\cap l_5$, we move $l_2$ in the pencil of lines through $A_1$. 
\item $A_1\notin l_{3,4}$ and $A_6\notin l_{3,4}$.\\
If $A_1\in l_{3,4}$, we move $l_{3,4}$ in the pencil through $A_5$
by moving $A_2$ on $l_1$.
\item $A_{3,4}\notin l_1$ and $A_{3,4}\notin l_6$.\\
Moving  $A_2$ and $A_3$ means that $A_{3,4}$ moves on $l_5\neq l_6$,
while moving  $A_4$ and $A_5$ makes $A_{3,4}$ to move on $l_2\neq l_1$.
\item $l_1\neq l_{3,4}$ and $l_6\neq l_{3,4}$.\\
This first condition means that $A_2$, $A_5$ and $A_n$ are not collinear,
and the second that $A_2$, $A_5$ and $A_7$ are not collinear. For $n=7$ 
these conditions coincide and are satisfied because $l_6\neq l_1$.
Let $n>7$ and suppose $A_5 \in l_1$. 
Then $A_5=l_1\cap l_4$ ($l_1\neq l_4$
as $A_2\notin l_4$).
We can move $A_5$ and $A_6$, moving  $A_5$ on $l_4$ off $l_1$.
If $A_2\in l_6$, then moving  $A_5$ on $l_4$ moves
$l_6$  in the pencil of lines through $A_7$.
\item $l_2\neq l_5$. \\ 
This holds as $A_3\notin l_5$.
\item $l_2 \nparallel l_{3,4}$ and $l_5 \nparallel l_{3,4}$.\\
If $l_5 \parallel l_{3,4}$ we
move $A_2$ and $A_3$, moving $A_2$ on $l_1$. As $A_5\notin l_1$
by a previous step,
this means that $l_{3,4}$ moves, whereas $l_5$ does not move.
If $l_2 \nparallel l_{3,4}$ we move $A_4$ and $A_5$.
\end{itemize}

The last condition to be satisfied is $A_{3,4}\notin l_{3,4}$.
If $A_{3,4}\in l_{3,4}$, 
then $l_{3,4}$, $l_2$ and $l_5$ are concurrent
and $A_{3,4}=B_3=B_4$. Now the conditions for the 
degenerate case of the 5-axes theorem (Theorem \ref{degen})
are satisfied. We find that $\bar g_2$, $\bar g_3$, $\bar g_4$
and $\bar g_5$ lie in a pencil. We conclude that $\bar g_5=g_5$
also in this case. 

We compute the image of $A_{3,4}$ under the involution on $l_5$
determined by $A_4$, $l_4$ and $g_5$, both
when  $A_{3,4}\in l_{3,4}$ and  $A_{3,4}\notin l_{3,4}$. 
According to the proof of Lemma 
\ref{invo} we have to intersect the line through $A_{3,4}$, parallel to
$l_4$ with $g_5$ and connect the intersection point with $A_4$. Then we
draw parallel to this last line a  line through $A_5$. 
The construction of the axis $\bar g_5=g_5$ shows that the line
through $A_4$ is parallel to $\lijn(A_2,A_5)$. 
Therefore the image of $A_{3,4}$ is $B_4=l_5\cap\lijn(A_2,A_5)$. 

If $A_{3,4}=B_4$, then it is a fixed point of the involution
and by moving
$A_2$ on $l_1$  and $A_3$ on $l_4$ 
we move $A_{3,4}$  on $l_5$, so that it is no longer
a fixed point of the involution, and therefore $A_{3,4}\neq B_4$,
giving $A_{3,4}\notin l_{3,4}$.

This shows that we can satisfy all assumptions.
For the new configuration with the same center $M$ we can conclude
by the induction hypothesis that also $g_1$, $g_{n}$ and $g_{n-1}$ 
pass through $M$. This then also holds for the original configuration.
\end{proof}

\end{document}